\documentclass[12pt]{amsart}

\usepackage[english]{babel}
\usepackage[utf8x]{inputenc}
\usepackage[T1]{fontenc}
\usepackage[all,cmtip]{xy}
\usepackage[margin=1.5in]{geometry}

\usepackage{amsmath}
\usepackage{amstext}
\usepackage{amsfonts}
\usepackage{amssymb}
\usepackage{amsthm}
\usepackage{mathtools}
\usepackage{dsfont}
\usepackage{color}
\usepackage{graphicx}
\usepackage[colorinlistoftodos]{todonotes}
\usepackage[colorlinks=true, allcolors=blue, pagebackref=true]{hyperref}
\usepackage{float}
\usepackage{cleveref}
\usepackage[alphabetic, initials]{amsrefs}

\usepackage[normalem]{ulem}

\DeclareFontFamily{U}{mathb}{\hyphenchar\font45}
\DeclareFontShape{U}{mathb}{m}{n}{
<-6> mathb5 <6-7> mathb6 <7-8> mathb7
<8-9> mathb8 <9-10> mathb9
<10-12> mathb10 <12-> mathb12
}{}
\DeclareSymbolFont{mathb}{U}{mathb}{m}{n}
\DeclareMathSymbol{\llcurly}{\mathrel}{mathb}{"CE}
\DeclareMathSymbol{\ggcurly}{\mathrel}{mathb}{"CF}

\allowdisplaybreaks

\newtheorem{theorem}{Theorem}[section]
\newtheorem{lemma}[theorem]{Lemma}
\newtheorem{prop}[theorem]{Proposition}

\newtheorem{cor}[theorem]{Corollary}
\newtheorem{thm}[theorem]{Theorem}
\newtheorem{lem}[theorem]{Lemma}

\newtheorem*{cor*}{Corollary}
\newtheorem*{thm*}{Theorem}
\newtheorem*{lem*}{Lemma}

\newtheorem*{prop*}{Proposition}

\theoremstyle{definition}

\newtheorem{defn}[theorem]{Definition}
\newtheorem{example}[theorem]{Example}

\newtheorem*{defn*}{Definition}

\theoremstyle{remark}
\newtheorem{remark}[theorem]{Remark}

\newcommand{\pr}{\operatorname{Prob}}

\newcommand{\act}{\curvearrowright}

\newcommand{\Op}{\underline{\bf{\rm OA}}}

\newcommand{\cB}{\mathcal{B}}

\newcommand{\cH}{\mathcal{H}}

\newcommand{\cP}{\mathcal{P}}

\newcommand{\cU}{\mathcal{U}}

\newcommand{\bC}{{\mathbb{C}}}
\newcommand{\bE}{{\mathbb{E}}}
\newcommand{\bF}{{\mathbb{F}}}

\newcommand{\bN}{{\mathbb{N}}}
\newcommand{\bZ}{{\mathbb{Z}}}

\newcommand{\bR}{{\mathbb{R}}}

\newcommand{\sA}{{\mathsf{A}}}
\newcommand{\sB}{{\mathsf{B}}}
\newcommand{\sC}{{\mathsf{C}}}
\newcommand{\sD}{{\mathsf{D}}}

\newcommand{\om}{\omega}

\newcommand{\id}{\operatorname{id}}

\newcommand{\Sub}{\operatorname{Sub}}

\newcommand{\G}{\Gamma}

\newcommand{\dist}{\mathbin{\mathaccent\cdot\cup}}

\newcommand{\inc}[2]{#1 \subset #2}

\newcommand{\incInt}[3]{#1 \subset #2 \subset #3}

\title[]{Tight inclusions of $C^*$-dynamical systems}

\author[Y. Hartman]{Yair Hartman}
\address{Yair Hartman\\ Ben-Gurion University of the Negev}
\email{hartmany@bgu.ac.il}

\author[M. Kalantar]{Mehrdad Kalantar}
\address{Mehrdad Kalantar\\ University of Houston}
\email{kalantar@math.uh.edu}

\date{}

\begin{document}

\begin{abstract}
We study a notion of tight inclusions of $C^*$- and $W^*$-dynamical systems which is meant to capture a tension between topological and measurable rigidity of boundary actions. An important case of such inclusions are $C(X)\subset L^\infty(X, \nu)$ for measurable boundaries with unique stationary compact models. We discuss the implications of this phenomenon in the description of Zimmer amenable intermediate factors.
Furthermore, we prove applications in the problem of maximal injectivity of von Neumann algebras.
\end{abstract}

\thanks{YH was partially supported by ISF grant 1175/18. MK was supported by a Simons Foundation Collaboration Grant (\# 713667).}

\maketitle

\section{Introduction}\label{sec:intro}

One of the key tools in rigidity theory is the notion of boundary actions in the sense of Furstenberg \cite{Furs63, Furs73}. These actions are defined in both topological and measurable setups, and exploiting their  dynamical and ergodic theoretical properties reveals various rigidity phenomena of the underlying groups.

For example, the fact that the measurable Furstenberg-Poisson Boundaries of irreducible lattices in higher rank semisimple Lie groups, have few quotients (\textit{Factor Theorem}) implies rigidity for normal subgroups (\textit{Normal Subgroup Theorem}), and a classification of certain spaces related to the Furstenberg-Poisson Boundary (\textit{Intermediate Factor Theorem}) implies rigidity of Invariant Random Subgroups. These rigidity phenomena are ``higher-rank phenomena'' either in the classical sense of semi-simple Lie groups, or for product groups, and are based on the measure theoretical boundary.
Recently, these properties were shown to imply 
also strong rigidity results in noncommutative settings (\cite{BH, BBHP}).

On the other front, dynamical properties of the topological boundaries have been shown to imply certain noncommutative rigidity properties, such as $C^*$-simplicity and the unique trace property (\cite{KK17, BKKO}). 

In many natural examples, measurable boundaries are concretely realized on topological boundaries, and one expects that this would reflect in their dynamical properties. 
However, the connection between the two notions of boundary actions has barely been systematically investigated.

A typical instance in which the interaction between the two setups arises is a topological boundary admitting a unique stationary measure turning it into a measurable boundary. 
A systematic study of such action in the noncommutative setting was undertaken in the authors' work \cite{HartKal}, wherein the framework of this connection, properties of measurable boundaries, were used in $C^*$-simplicity problems.

An important consequence of having a unique stationary boundary measure is a uniqueness property for equivariant maps from the space of continuous functions into the space of essentially bounded measurable functions on the boundary.
This work is devoted to the study of this particular uniqueness phenomenon and several of its applications. As we will see below, this  property is not an exclusive feature of certain boundary actions, and it does appear in setups with quite different behavior.

More precisely, this work is around the following notion.
Given a locally compact second countable (lcsc) group $G$, we denote by $\Op_G$ the category of all unital $G$-$C^*$-algebras and $G$-$W^*$-algebras where the morphisms are $G$-equivariant ucp maps. 
\begin{defn*}[Definition \ref{def:tight}]
We say a $C^*$-inclusion $\inc{\sA}{\sB}$ of objects $\sA, \sB\in \Op_G$ is \emph{$G$-tight} if the inclusion map is the unique $G$-equivariant ucp map from $\sA$ to $\sB$.
\end{defn*}
This property has already been exploited in some previous work. 
When $\sA$ is a commutative $C^*$-algebra, and $\sB$ a commutative von Neumann algebra, this coincides with Furman's notion of alignment systems \cite{Furm08}, a key concept in his work on rigidity of homogeneous actions of semisimple groups.
Around the same time, in a completely different context, Ozawa \cite{Oza07} proved that for a quasi-invariant and doubly-ergodic measure $\nu$ on the Gromov boundary $\partial \bF_n$ of the free group, the inclusion $C(\partial \bF_n)\subset L^\infty(\partial \bF_n, \nu)$ is $\bF_n$-tight. He used this property to prove a nuclear embedding result for the reduced $C^*$-algebra of the free group.

As mentioned earlier, we have the following fact.
\begin{thm*}[Theorem \ref{thm:USB->rigid}]
Let $G$ be a lcsc group, and let $\mu\in\pr(G)$ be an admissible  probability measure on $G$. Suppose $X$ is a minimal compact $G$ space that admits a unique $\mu$-stationary probability
$\nu$ such that $(X,\nu)$ is a $\mu$-boundary. Then the canonical embedding $C(X)\subset L^\infty(X, \nu)$ is $G$-tight.
\end{thm*}

The tightness property becomes particularly fruitful in combination with the notion of Zimmer amenability. 

\begin{thm*}[Theorem \ref{thm:Zamen-int-fact}]
Let $G$ be a lcsc group and let $\mu\in\pr(G)$ be an admissible  measure such that the Furstenberg-Poisson Boundary $(B, \nu)$ of $(G, \mu)$ has a uniquely stationary compact model. Let $(Y, \eta)$ be a $(G, \mu)$-space and let $\phi\colon (\tilde Y, \tilde \eta)\to (Y, \eta)$ be the standard cover in the sense of Furstenberg--Glasner. 
If $(\tilde Y, \tilde \eta) \xrightarrow{\varphi} (Z,\om) \xrightarrow{\psi} (Y, \eta)$ are measurable $G$-maps such that $\psi\circ\varphi=\phi$ and $(Z,\om)$ is Zimmer amenable, then $(\tilde Y, \tilde \nu) \overset{\varphi}{\cong} (Z,\om)$.
\end{thm*}
For discrete groups $\G$, we prove a noncommutative version of this. Namely, we show that under the same conditions, there are no injective von Neumann algebras $M$ satisfying $\G\ltimes L^\infty(Y, \eta)\subseteq M \subsetneq \G\ltimes L^\infty(\tilde Y, \tilde \eta)$ (Corollary~\ref{cor:no-inj-int-two-case}).

\medskip

 Examples of tight inclusions involving noncommutative $C^*$-algebras include the embedding of tight $\G$-$C^*$-algebras in their associated crossed products.
This, for instance, yields the following maximal injectivity result.

\begin{thm*}[Corollary~\ref{cor:max-inj-vn}]
Let $\G$ be a discrete group and $\mu\in\pr(G)$ a generating measure such that the Furstenberg-Poisson Boundary $(B, \nu)$ of $(\G, \mu)$ has a uniquely stationary compact model. Let $\G\curvearrowright (Z, m)$ be a measure-preserving action. 
Then the von Neumann algebra $\G\ltimes L^\infty(B \times Z , \nu \times m)$ is maximal injective in $\G\ltimes \big(\cB(L^2(B, \nu))\overline{\otimes}L^\infty(Z, m)\big)$.
\end{thm*}

A special case of the above in which $\G$ is an irreducible lattice in a centerless higher rank lattice, and $Z$ is trivial, yields a stronger conclusion than Suzuki's maximal injectivity result \cite[Corollary 3.8]{Suz18}. In particular, it shows that this maximal injectvity is not a higher rank phenomenon but rather follows from the wider framework of tightness. 
Consequently, this provides a large class of new examples of maximal injective von Neumann algebras (see comments after Corollary~\ref{cor:max-inj-vn}).

However, as shown in Corollary \ref{cor:rig-not->factors} and Theorems \ref{thm:specgapprop-equiv} and \ref{thm:parsubgrp->rigmsre}, this notion of tightness is not bound to only certain boundary actions, it is more general even in the commutative case, and there are examples of tight actions with properties far from boundary actions.\\


Next, we fix our notation and briefly  review some of the definitions and basic facts that will be used in the rest of the paper.

Throughout the paper, unless otherwise stated, $G$ is a lcsc group and $\G$ denotes a countable discrete group. 
We write $G\act X$ to mean a continuous action of $G$ on a compact Hausdorff space $X$ by homeomorphisms. In this case we say $X$ is a compact $G$-space.
Given $G\act X$ and $G\act Y$, we say $Y$ is a ($G$-)factor of $X$, or that $X$ is a ($G$-)extension of $Y$, if there is a continuous map $\varphi$ from $X$ onto $Y$ that is $G$-equivariant, that is $\varphi(gx)=g\varphi(x)$ for all $g\in G$ and $x\in X$.

Given $G\act X$ and $\nu\in \pr(X)$, the {Poisson transform} (associated to $\nu$) is the map $\cP_\nu\colon  C(X) \to C_b(G)$ defined by $\cP_\nu(f) (g) = \int_X f(gx)d\nu(x)$ for $f\in C(X)$ and $g\in G$, where $C_b(G)$ denotes the space of all bounded continuous functions on $G$. This is obviously a continuous unital linear positive map which is also $G$-equivariant.

We use similar standard terminology for measurable actions of a lcsc group $G$. All measure spaces considered here are assumed to be standard Borel probability spaces. 
We write $G\act (X, \nu)$ to mean a measurable action of $G$ on a standard Borel probability space $(X, \nu)$ by measure isomorphisms; in particular, in this setup, the measure $\nu$ is quasi-invariant, that is, $g\nu\sim \nu$ for any $g\in G$. In this case we say $(X, \nu)$ is a  probability $G$-space.
Given $G\act (X, \nu)$ and $G\act (Y, \eta)$, we say $(Y, \eta)$ is a ($G$-)factor of $(X, \nu)$, or that $(X, \nu)$ is a ($G$-)extension of $(Y, \eta)$, if there is a $G$-equivariant measurable map $\varphi$ from $X$ to $Y$ such that $\eta=\varphi_*\nu$.
In this case, the map $\varphi$ yields a canonical $G$-equivariant von Neumann algebra embedding $\varphi^*: L^\infty(Y, \eta)\to L^\infty(X, \nu)$.

For an action $G\act (X, \nu)$, similarly to the continuous case, we denote by $\cP_\nu\colon  L^\infty(X) \to L^\infty(G)$ the Poisson transform $\cP_\nu(f) (g) = \int_X f(gx)d\nu(x)$, which is a $G$-equivariant normal unital positive linear map.

Let $G\act (X, \nu)$, and $\mu\in\pr(G)$. We say $\nu$ is $\mu$-stationary if $\mu*\nu = \nu$, where $\mu*\nu$ is the convolution of the measures. In this case we write $(G,\mu)\act (X, \nu)$ and say $(X, \nu)$ is a $(G,\mu)$-space.

A compact model for a probability $G$-space $(Y, \eta)$ is a compact $G$-space $X$ and a quasi-invariant $\nu\in\pr(X)$ such that $(X, \nu)$ is isomorphic to $(Y, \eta)$ as probability $G$-spaces.

We also consider noncommutative actions in this paper, namely, actions of $G$ on $C^*$ and von Neumann algebras. All $C^*$-algebras considered here are assumed to be unital. By a $G$-$C^*$-algebra (respectively, $G$-von Neumann algebra) we mean a unital $C^*$-algebra $\sA$ on which $G$ acts continuously by $*$-automorphisms. Continuity in this setup is with respect to the point-norm topology (respectively, point-weak* topology) of $\sA$. 

Given a discrete group $\G$ and a $\G$-$C^*$ or $\G$-von Neumann algebra $\sA$, a \emph{covariant representation} of $(\G, \sA)$ is a pair $(\pi,\rho)$, where $\pi$ is a unitary representation of $\G$ on Hilbert space $\cH_\pi$, and $\rho\colon \sA\to \cB(\cH_\pi)$ is a $\G$-equivariant representation of $\sA$, where $\G\act \cB(\cH_\pi)$ by inner automorphisms ${\rm Ad}_{\pi(g)}$, $g\in\G$. In this case, we let $\G\ltimes_{\scriptscriptstyle\tiny\pi}^{\scriptscriptstyle\tiny\rho}\sA$ be the $C^*$ or von Neumann algebra generated by the set $\{\rho(a)\pi(g):a\in \sA, g\in \G\}$, depending on $\sA$. 
In the case of regular covariant representation, we use the notation $\G\ltimes\sA$ for either the reduced $C^*$-crossed product or the von Neumann algebra crossed product of the action, again depending on $\sA$.

We refer the reader to \cite{BrownOz} for the definitions and details concerning these constructions and their properties.

We often in this paper consider equivariant ucp maps from $C^*$-algebras into von Neumann algebras. In view of the following known fact, these maps should be considered as  noncommutative counterparts of \emph{quasi-factor maps} in the sense of Glasner~\cite[Chapter 8]{glasner2003ergodic}. We omit the proof of the following, which can be found in \cite{Oza07}. The latter argument was written for a special example, but as noted in \cite{BasRadul20}, the same argument works in general; (see also \cite[Proposition 4.10]{BBHP}).
\begin{lem}\label{lem:quasi-factor}
Let $X$ be a compact metric ($G$-space), and $(Y, \eta)$ a standard probability ($G$-)space. Then there is a one-to-one correspondence between ($G$-equivariant) $\eta$-measurable maps $Y\to \pr(X)$ and equivariant ucp maps $C(X) \to L^{\infty}(Y,\eta)$.
\end{lem}

\subsection*{Boundary actions}
The most natural examples in our context are topological and measurable boundary actions in the sense of Furstenberg. We briefly review the definitions and refer the reader to \cite{Furs73, Glasner-book, Furm-book, HartKal} for more details.

An action $G\act X$ is said to be a {topological boundary action}, and $X$ is said to be a topological $G$-boundary, if for every $\nu\in \pr(X)$ and $x\in X$ there is a net $g_i$ of elements of $G$ such that $g_i \nu \to \delta_x$ in the weak* topology, where $\delta_x$ is the Dirac measure at $x$.
It was shown by Furstenberg \cite[Proposition 4.6]{Furs73} that any lcsc group $G$ admits a unique (up to $G$-equivariant homeomorphism) maximal $G$-boundary $\partial_F G$ in the sense that every $G$-boundary $X$ is a $G$-factor of $\partial_F G$.

Measurable boundary actions are defined as follows. Let $\mu\in\pr(G)$ be an admissible measure, that is, $\mu$ is absolutely continues with respect to the Haar measure and is not supported in a proper closed subgroup, and let $\nu$ be a $\mu$-stationary probability measure on a metrizable $G$-space $X$. The action $(G, \mu)\act (X,\nu)$ is a {$\mu$-boundary action} if for almost every path $(\omega_k)\in G^{\bN}$ of the $(G, \mu)$-random walk, the sequence $\omega_k\nu$ converges to a Dirac measure $\delta_{x_\om}$.

An action $G\act (Y, \eta)$ is said to be a $\mu$-boundary action if it has compact model which is a $\mu$-boundary in the above sense.
Similarly to the topological case, there is a unique (up to $G$-equivariant measurable isomorphism) maximal $\mu$-boundary $(B , \nu)$, called the {Furstenberg-Poisson Boundary} of the pair $(G,\mu)$, in the sense that every $\mu$-boundary $(X,\eta)$ is a $G$-factor of $(B , \nu)$.

As mentioned in the introduction, the case of measurable $\mu$-boundaries with compact models supporting unique stationary measures, are of particular interest for us in this work. Following \cite[Definition 3.9]{HartKal}, we refer to such an action a $\mu$-USB.
The study of these systems was indeed initiated by Furstenberg in~\cite{Furs73}, where they were called {$\mu$-proximal actions}. These were further studied in \cite{Margulis-book-91, Glasner-Weiss-16}.

\section{Tight inclusions}\label{sec:tight-incl}
Let $G$ be a lcsc group. We denote by $\Op_G$ the category of all unital $G$-$C^*$-algebras and $G$-von Neumann algebras where the morphisms are $G$-equivariant ucp maps (not assumed normal even between von Neumann algebras).

If $\G$ is a countable discrete group, every $\G$-von Neumann algebra is also a $\G$-$C^*$-algebra, and therefore $\Op_{\G}$ is just the category of all unital $\G$-$C^*$-algebras.

Given $\sA\in\Op_{\G}$, we say $\sA$ is $\G$-injective if it is an injective object in the category $\Op_{\G}$. We refer the reader to \cite{Ham85} and \cite{KK17} for more details on this concepts and its connection to boundary actions.

A unital $C^*$-algebra $\sA$ is injective if it is injective in the category of unital $C^*$-algebras with ucp maps as morphisms, equivalently, $\G$-injective for the trivial group $\G=\{e\}$.

Note that all positive maps between commutative $C^*$-algebras are automatically completely positive, but for the sake of consistency, we keep assuming ucp for all our morphisms.

\begin{defn}\label{def:tight}
An inclusion $\inc{\sA}{\sB}$ of objects $\sA, \sB\in \Op_G$ is called \emph{$G$-tight} if the inclusion map is the unique $G$-equivariant ucp map from $\sA$ to $\sB$.
\end{defn}

We begin with few observations on the general properties of $G$-tight inclusions before getting to some basic examples. 

\begin{prop}\label{prop:tight->suablg->tight}
Suppose $\inc{\sA}{\sB}$ is a $G$-tight inclusion of objects in $\Op_G$. Then for any $\sC\in\Op_G$ with $\sA\subset \sC\subset \sB$, the inclusion $\inc{\sA}{\sC}$ is $G$-tight; in particular, $\inc{\sA}{\sA}$ is $G$-tight.
\begin{proof}
Every $G$-equivariant ucp map $\sA \to \sC$ is in particular a $G$-equivariant ucp map from $\sA$ to $\sB$. Thus $G$-tightness of $\inc{\sA}{\sB}$ implies $G$-tightness of $\inc{\sA}{\sC}$.
\end{proof}
\end{prop}
We refer to an object $\sA$ for which $\inc{\sA}{\sA}$ is $G$-tight as a ($G$-)self-tight object.

\begin{lem}\label{lem:rigincl+inv->trv}
Suppose $\inc{\sA}{\sB}$ is a $G$-tight inclusion of objects in $\Op_G$. If $\sA$ admits a $G$-invariant state, then $\sA=\bC$.
\end{lem}
\begin{proof}
Suppose $\tau$ is a $G$-invariant state on $\sA$. Then $\tau$ is, in particular, a $G$-equivariant ucp map from $\sA$ to $\bC\subset\sB$. Hence $\tau=\id$ by $G$-tightness of the $\inc{\sA}{\sB}$, and this implies $\sA=\bC$.
\end{proof}

Recall that a lcsc group $G$ is said to be amenable if any continuous action of $G$ by 
affine homeomorphisms on a compact convex subset of a topological vector space has a fixed point.
\begin{cor}
If there is a $G$-tight inclusion $\inc{\sA}{\sB}$ of objects in $\Op_G$ with $\sA$ non-trivial, then $G$ is non-amenable.
\end{cor}
\begin{proof}
Suppose $\inc{\sA}{\sB}$ is a $G$-tight inclusion of objects in $\Op_G$. If $G$ is amenable, then $\sA$ admits a $G$-invariant state, and therefore we have $\sA=\bC$ by Lemma~\ref{lem:rigincl+inv->trv}. This implies the claim.
\end{proof}

Recall that every discrete group $\G$ admits a maximal normal subgroup $N$ such that $\G/N$ is an ICC group (every non-trivial element has infinite conjugacy class). This subgroup is called the hyper-FC-center of $\G$. 
\begin{prop}
Let $\inc{\sA}{\sB}$ be a $\G$-tight inclusion of objects in $\Op_{\G}$. Then the hyper-FC-center of $\G$ is contained in the kernel of the action of $\G$ on $\sA$. 
\end{prop}

\begin{proof}
Suppose $g\in \G$ has finite conjugacy class $C_g$. Define $\Phi_g\colon  \sA \to \sA$ by
\[ ~~~~~~~~ \Phi_g(a) = \frac{1}{\#C_g}\sum_{k\in C_g} ka ~~~~~~~~ (a\in \sA) .\]
Then $\Phi_g$ is ucp, and for every $h\in G$ and $a\in\sA$,
\[\begin{split}
h\Phi_g(a) &= \frac{1}{\#C_g}\sum_{k\in C_g} hka = \frac{1}{\#C_g}\sum_{k\in C_g}kha = \Phi_g (ha) ,
\end{split}\]
which shows $\Phi_g$ is equivariant. Thus, by tightness $\Phi_g=\id$. Since $a\mapsto ka$ is a *-isomorphism of $\sA$ for each $k\in C_g$, it follows $ka = a$ for every $k\in C_g$ and $a\in\sA$. In particular, $g$ acts trivially on $\sA$. This implies the normal subgroup $N$ of $G$ consisting of all finite conjugacy elements lie in the kernel of the action $G\act \sA$. Repeating the argument for the action $G/N\act \sA$, and taking a transfinite induction yields the result.
\end{proof}

In particular, we conclude the following for discrete groups with faithful actions on tight inclusions.
\begin{cor}
Let $\inc{\sA}{\sB}$ be a $\G$-tight inclusion of objects in $\Op_{\G}$. If the action of $\G$ on $\sA$ is faithful, then $\G$ is an ICC group. 
\end{cor}

It follows from the definition that for any $\sB\in\Op_{G}$, the inclusion $\inc{\bC}{\sB}$ is tight.
\begin{prop}
Let $\sB\in \Op_{G}$. Then $\sB$ admits a maximal subalgebra $\sA$ such that $\inc{\sA}{\sB}$ is tight.
\begin{proof}
This follows from a standard Zorn's Lemma argument combined with the above observation that the inclusion $\inc{\bC}{\sB}$ is $\G$-tight.
\end{proof}
\end{prop}

We continue with some basic examples of $G$-tight inclusions. More examples will be studied in later sections.

\begin{example}
For every countable discrete group $\G$, $\bC$ is a maximal subalgebra of $\ell^{\infty}(\G)$ which is tight. Indeed, since the right translation by any element of $\G$ is a ucp equivariant map on $\ell^{\infty}(\G)$, any function in a tight subalgebra of $\ell^{\infty}(\G)$ must be invariant under right translation by every $g\in\G$, hence constant.\qed
\end{example}

While we are mainly interested in the setup in which one object is a von Neumann algebra and the other is a $C^*$-algebra, the following is an example that is purely $C^*$-algebraic. 
\begin{example}\label{top_bondaries}
If $Y$ is a topological $G$-boundary, and $X$ is a continuous $G$-factor of $Y$, then it follows from \cite[Proposition 4.2]{Furs73} that the inclusion $\inc{C(X)}{C(Y)}$ is $G$-tight.

In particular, $C(X)$ is self-tight for every topological boundary $X$.\qed 
\end{example}

\begin{example}
Let $\cH$ be a Hilbert space, and let $G=\cU(\cH)$ be the group of all unitaries on $\cH$, considered as a (uncountable) discrete group. The space $\cB(\cH)$ of all bounded linear maps on $\cH$ is canonically a $G$-$C^*$-algebra. We show $\cB(\cH)$ is $G$-self-tight. For this, let $\psi$ be a $G$-equivariant ucp map on $\cB(\cH)$, and let $p$ be a projection on $\cH$. Let $H\leq G$ be the group of unitaries on $\cH$ commuting with $p$. Then by equivariance, $\psi(p)$ also commutes with $H$, hence $\psi(p)\in\{p\}''={\rm span}\{p, 1_{\cB(\cH)}-p\}$. Thus, by positivity, $\psi(p)=r_pp+s_pp^\perp$ for some $r_p, s_p\in \bR^+$. 

Now, given two projections $p$ and $q$ such that both their ranges and orthogonal subspaces to their ranges are infinite dimensional, there is $u\in G$ such that $upu^*=q$ and $up^\perp u^*=q^\perp$. Then by equivariance, $\psi(q) = \psi(upu^*)= u\psi(p)u^*= u(r_pp+s_pp^\perp)u^*= r_pq+s_pq^\perp$, therefore $r_p=r_q(=:r)$ and $s_p=s_q(=:s)$. In particular, in this case, since $\psi$ is unital, $1-s_p=r_{p^\perp}=r_p$. Moreover, given such $p$ as above, if we choose subprojections $p_1$ and $p_2$ of $p$ with infinite dimensional ranges such that $p=p_1+p_2$, then we get 
\[\begin{split}
r{p_1}+s{p_1}^\perp+r{p_2}+s{p_2}^\perp =&
\psi(p_1+p_2) = \psi(p) = rp+sp^\perp \\=& 
r{p_1}+r{p_2}+s{(p_1+p_2)}^\perp \\=&
r{p_1}+r{p_2}+s{p_1}^\perp+s{p_2}^\perp-s1_{\cB(\cH)} ,
\end{split}\]
which implies $s=0$, hence $r=1$, and $\psi(p)=p$ for all projections $p$ as above. Now, if $q$ is a finite rank projection, then $q\leq p$ for some projection $p$ as above. Then, since both $p-q$ and $(p-q)^\perp$ have infinite ranks, we get $p=\psi(p) = \psi(p-q)+ \psi(q) = p-q + \psi(q)$, which implies $\psi(q) = q$. Hence, $\psi$ restricts to the identity map on projections, and since the set of projections span a norm dense subspace of $\cB(\cH)$, we conclude $\psi=\id$.\qed
\end{example}
Rigidity properties of locally compact groups are usually passed down to their lattices. This is the case for the tightness condition.
\begin{thm}\label{thm:lattice}
Let $G$ be a lcsc group and let $\G$ be a lattice in $G$. Then any $G$-tight inclusion $\inc{\sA}{\sB}$ of objects in $\Op_G$ is $\G$-tight.
\begin{proof}
Let $\G$ be a lattice in $G$ and suppose $\inc{\sA}{\sB}$ is a $G$-tight inclusion of objects in $\Op_G$. Let $\psi\colon  \sA\to \sB$ be a $\G$-equivariant ucp map. For each $a\in \sA$, the map $g \mapsto g\psi(g^{-1}a)$ is continuous from $G$ to $\sB$ and restricts to the constant function $\gamma\mapsto \psi(a)$ on $\G$. Thus, it induces a continuous function $\rho_a\colon  G/\G\to\sB$, $g\G\mapsto g\psi(g^{-1}a)$. Define the map $\phi\colon \sA\to \sB$, by $\phi(a) := \int_{G/\G} \rho_a(g\G) d\mu(g\G)$, where $\mu$ is a $G$-invariant Borel probability on $G/\G$. Then $\phi$ is ucp, and for every $h\in G$,
\[\begin{split}
h\phi(a) &= \int_{G/\G} h\rho_a(g\G) d\mu(g\G) = \int_{G/\G} hg\psi(g^{-1}a) d\mu(g\G) \\&= \int_{G/\G} g\psi(g^{-1}ha) d\mu(g\G) = \phi(ha) ,
\end{split}\]
which shows $\phi$ is $G$-equivariant. Hence, by $G$-tightness, $\phi$ is the inclusion map. Since the maps $\sA\ni a\mapsto g\psi(g^{-1}a)\in \sB$ are ucp for every $g\in G$, by extremality of the inclusion map in the set of ucp maps, it follows $g\psi(g^{-1}a) = a$ for all $a\in \sA$ and a.e.\ $g\G\in G/\G$. By continuity of the action and the fact that $\mu$ has full support, we conclude the latter equality for all $g\in G$.
Hence, in particular, $\psi = \id_\sA$.
\end{proof}
\end{thm}
The following lemma, which is essentially \cite[Lemma 3.3]{Ham85}, provides a useful technical tool for proving the tightness of inclusions.
\begin{lem}\label{lem:faith-cond-exp-->G-rigid}
Let $\incInt{\sA}{\sB}{\sC}$ be inclusions of objects in $\Op_G$ such that
\begin{enumerate}
\item
the inclusion $\inc{\sA}{\sB}$ is $G$-tight, and 
\item
there is a faithful $G$-equivariant conditional expectation $\bE$ from $\sC$ onto $\sB$;
\end{enumerate}
then the inclusion $\inc{\sA}{\sC}$ is also $G$-tight.
\end{lem}

\begin{proof}
If $\psi\colon \sA\to \sC$ is a $G$-equivariant ucp map, then $\bE\circ\psi$ is a $G$-equivariant ucp map from $\sA$ to $\sB$. Hence, by the tightness of $\inc{\sA}{\sB}$, we get that $\bE\circ\psi = \id$.  
Then using the fact that $\bE$ is a $\sB$-bimodule map, and applying the Schwarz inequality for the completely positive map $\psi$, we see for every $x\in \sA$ that $\bE\big((x^*-\psi(x^*))(x-\psi(x))\big) = 0$, which by faithfulness of $\bE$, it implies $\psi(x) =x$.
Thus, the inclusion $\inc{\sA}{\sC}$ is $G$-tight.
\end{proof}

Recall our notation that $\G\ltimes\sB$ denotes either the reduced $C^*$-crossed product, or the von Neumann crossed product for $\sB\in\Op_{\G}$.
\begin{cor}\label{cor:tight->crsprdt-tight}
Let $\G$ be a discrete group and $\inc{\sA}{\sB}$ a $\G$-tight inclusion of objects in $\Op_{\G}$. Then the canonical inclusion $\inc{\sA}{\G\ltimes\sB}$ is also $\G$-tight.
\end{cor}

\begin{proof}
The canonical conditional expectation $\bE_0 \colon  \G\ltimes\sB \to \sB$ is $\G$-equivariant and faithful. Hence the result follows from Lemma~\ref{lem:faith-cond-exp-->G-rigid}.
\end{proof}


\subsection{(Weak) Zimmer amenability}\label{sec:WZamen}

In this section, we recall the notion of weak Zimmer amenability and examine its combination with tight inclusions. This weaker notion of Zimmer amenability is indeed enough for many applications. Furthermore, it has the advantage that it can be extended to $C^*$-algebra setting as well, at least in the discrete group case.

\begin{defn}
An object $\sB \in \Op_G$ is said to be \emph{weakly Zimmer amenable} if for every $\sA\in\Op_G$ there is a $G$-equivariant ucp map $\sA\to \sB$.
\end{defn}

For any locally compact group $G$, the space $C_b^{lu}(G)$ of left-uniformly continuous functions on $G$ is a weakly Zimmer amenable $G$-$C^*$-algebra. 

Clearly, if $G$ is amenable, then every object $\sB \in \Op_G$ is weakly Zimmer amenable.

If $\G\act (X, \nu)$ is a Zimmer amenable probability measure-class preserving action of a discrete group $\G$, then $L^\infty(X, \nu)$ is a weakly Zimmer amenable $\G$-$C^*$-algebra. 
In particular, if $\mu\in\pr(\G)$, and $(B,\nu)$ is the Furstenberg-Poisson Boundary of $(\G,\mu)$, then $L^{\infty}(B,\nu)$ is weakly Zimmer amenable.

More generally, since our objects are all unital, every $\G$-injective $\sA\in\Op_\G$ is weakly Zimmer amenable; the converse is not true: if $\G$ is non-amenable then one can see that $\cB(\ell^2(\G))$ is not $\G$-injective, but it is obviously weakly Zimmer amenable.

The following is a useful characterization of weakly Zimmer amenable $\G$-$C^*$-algebras in the case of discrete groups $\G$. Recall that by \cite[Theorem 3.11]{KK17}, $C(\partial_F\G)$ is $\G$-injevtive for any discrete group $\G$.
\begin{prop}\label{prop:furst-embed-->weak-zamen}
Let $\G$ be a countable discrete group. A $\G$-$C^*$-algebra $\sB$ is weakly Zimmer amenable iff $C(\partial_F \G)\subseteq \sB$ as a $\G$-invariant operator subsystem.
\begin{proof}
Let $\sB$ be a $\G$-$C^*$-algebra. Suppose $\sB$ is weakly Zimmer amenable. Then there is a $\G$-equivariant ucp map $C(\partial_F \G)\to \sB$. But any such map is isometric (e.g.\, \cite{KK17}), hence an embedding of $C(\partial_F \G)$ into $\sB$ as a $\G$-invariant operator subsystem.

Conversely, suppose $C(\partial_F \G)\subseteq \sB$ as a $\G$-invariant operator subsystem, and let $\sA\in\Op_{\G}$. Then by $\G$-injectivity of $C(\partial_F \G)$, there is a $\G$-equivariant ucp map from $\sA$ to $C(\partial_F \G)$, and so in $\sB$. Hence, $\sB$ is weakly Zimmer amenable.
\end{proof}
\end{prop}

\begin{prop}\label{prop:}
Let $\inc{\sA}{\sB}$ be a $\G$-tight inclusion of objects in $\Op_{\G}$. If $\sB$ is weakly Zimmer amenable, then $\sA=C(X)$ for some topological boundary $X$.
\end{prop}

\begin{proof}
Since $\sB$ is weakly Zimmer amenable, we have $C(\partial_F \G)\subseteq \sB$ as a $\G$-invariant operator subsystem by Proposition \ref{prop:furst-embed-->weak-zamen}. By $\G$-injectivity, there is a $\G$-equivariant ucp idempotent $\psi\colon  \sB\to C(\partial_F \G)$. By tightness of the inclusion $\inc{\sA}{\sB}$ the restriction of $\psi$ to $\sA$ is the identity map, which implies $\sA\subseteq C(\partial_F \G)$ (as a $\G$-operator subsystem). 

Since $C(\partial_F \G)$ is an injective $C^*$-algebra, its multiplication coincides with the Choi-Effros product associated to $\psi$. Since $\sA$ is a subalgebra of $\sB$ it follows that this product agrees with the original product on $\sA$. It follows that $\sA$ is indeed a subalgebra of $C(\partial_F \G)$, and so  of the form $\sA=C(X)$ for some topological $\G$-boundary $X$.
\end{proof}


\section{Tight inclusions in commutative setting: tight measure-classes}\label{sec:commutative}

In this section, we focus our attention on a special case of tight inclusions in the commutative setting. More precisely:

\begin{defn}\label{def:rigid-measure}
Let $X$ be a compact $G$-space. A non-singular probability measure $\nu\in\pr(X)$ is said to be \emph{$G$-tight} (or just tight if the group $G$ is clear from the context) if it has full support and the canonical embedding $\inc{C(X)}{L^\infty(X, \nu)}$ is a $G$-tight inclusion.
\end{defn}

In~\cite{Furm08}, Furman introduced and studied the notion of {\it alignment property}: given a measurable $G$-space $(X, \nu)$ and a compact $G$-space $Z$, a Borel measurable $G$-equivariant map $\pi\colon X\to Z$ is said to have the alignment property if the only Borel measurable $G$-equivariant map from $(X,\nu)$ to $\pr(Z)$ is the one given by $x\mapsto \delta_{\pi(x)}$.

The correspondence between maps $X\to \pr(Z)$ and ucp maps $C(Z)\to L^\infty(X,\nu)$ (Lemma~\ref{lem:quasi-factor}) implies that for $X$ and $\nu\in\pr(X)$ as in Definition~\ref{def:rigid-measure}, $\nu$ is $G$-tight iff the identity map $\id\colon  (X,\nu)\to X$ has the alignment property in the sense of Furman. 

\begin{remark}
Observe that tightness is, in fact, a property of the measure-class rather than a single measure. Furthermore, it can be considered as a property of the algebra $L^\infty(X,\nu)$: the existence of an $L^1$-dense, $L^\infty$-closed subalgebra of $L^\infty(X,\nu)$ with unique equivariant ucp map into $L^\infty(X,\nu)$. 
\end{remark}

Let us restate in the case of tight measure-classes the facts proven in Section~\ref{sec:tight-incl} for general tight inclusions.

\begin{prop}\label{prop:rig-meau->gen-props}
Let $G$ be a lcsc group.
\begin{enumerate}
\item
If a compact $G$-space $X$ admits both a tight probability measure and an invariant probability measure, then $X$ is a singleton.
In particular, for any lcsc group $G$ the only finite $G$-space that admits a tight measure-class is the trivial one.
\item
If $G$ admits a non-trivial action with a tight measure-class, then $G$ is non-amenable.
\item
If $\G$ is discrete and admits a faithful action on a compact space $X$ supporting a tight measure-class, then $\G$ is ICC.
\item
If $\G$ is a lattice in $G$, then any $G$-tight measure-class on any compact $G$-space is $\G$-tight.
\end{enumerate}
\end{prop}

A natural source of tight measure-classes is the following.

\begin{thm}\label{thm:USB->rigid}
Let $G$ be a lcsc group, and let $\mu\in\pr(G)$ be an admissible  probability on $G$. Suppose $X$ is a minimal compact $G$ space that admits a unique $\mu$-stationary probability
$\nu$ such that $(X,\nu)$ is a $\mu$-boundary. Then the measure-class of $\nu$ is $G$-tight.
\end{thm}

\begin{proof}
By Lemma~\ref{lem:quasi-factor} any $G$-equivariant ucp map from $C(X)$ to $L^\infty(X, \nu)$ corresponds to a measurable equivariant map from $(X, \nu)$ to $\pr(X)$. By \cite[Corollary 2.10(a)]{Margulis-book-91}, any such map is mapped into delta measures, hence a $G$-map on $(X, \nu)$. But the identity is the unique measurable $G$-map on $(X, \nu)$ (see e.g. Lemma~\ref{lem:normal-self-tight} below). Hence, the identity is the unique $G$-equivariant ucp map from $C(X)$ to $L^\infty(X, \nu)$. 
\end{proof}

Recall that topological boundaries yield self-tight algebras of continuous functions (Example~\ref{top_bondaries}). The same statement fails for general measurable boundaries. Namely, $L^\infty(B,\nu)$ is not self-tight and Theorem~\ref{thm:USB->rigid} shows that the tightens holds once restricting to a USB (if there is such). Another form of tightness that holds for measurable boundaries is the following well-known fact.
\begin{lemma}\label{lem:normal-self-tight}
Let $(B,\nu)$ be a $\mu$-boundary (for some $\mu\in\pr(G)$, where $G$ is a lcsc group). Then the only \textit{normal} equivariant unital positive $L^\infty(B,\nu) \to L^\infty(B,\nu)$ is the identity.
\begin{proof}
This follows along the same lines as in the proof of Theorem~\ref{thm:USB->rigid}, using the fact that there is no stationary probability measure other than $\nu$ that is absolutely continuous with respect to $\nu$ (this follows from the ergodicity of $\nu$, see e.g.\ \cite{BadShal06}).
\end{proof}
\end{lemma}

The lack of endomorphisms of boundaries is an important feature in both the measurable and the topological setups. Example~\ref{top_bondaries} and Lemma~\ref{lem:normal-self-tight} provide straightening of these results: not only there are no equivariant maps in the level of the points, but not even equivariant unital positive linear maps in the level of functions.

Theorem~\ref{thm:USB->rigid} already provides a vast class of examples, especially in the case of discrete groups. There is a significant amount of work on realizations of Furstenberg-Poisson Boundaries on concrete topological spaces, where the main tool is the strip criterion of Kaimanovich~\cite{Kaim00}. In many cases,
the topological space is compact, and it is proven 
that the Furstenberg-Poisson measure is the unique stationary measure on the discussed space. These cases include actions of linear groups on flag varieties~\cites{Ledrappier-85, Kaim00, Brofferio-Schapira-11}, hyperbolic groups acting on the Gromov boundary~\cite{Kaim00}, non-elementary subgroups of mapping class groups acting on the Thurston boundary~\cite{KaimMas96}, among others.

\subsection*{Tight measure-classes vs.\ topological boundaries} 
The arguments in \cite{Oza07} imply that Zimmer amenable tight actions of discrete groups are topological boundaries. This was noted in \cite[Theorem 2.3]{BasRadul20}, although the authors state the result under topological amenability, which is a stronger assumption. These arguments do not directly generalize to non-discrete groups, as the existence of the faithful equivariant conditional expectation is a key point in the argument, and such conditional expectation does not exist in general non-discrete cases.

Below, we give a simple alternative  proof of this fact for general locally compact second countable groups.

\begin{prop}[cf. \cite{Oza07} and \cite{BasRadul20}]\label{prop:Oza}
Let $G$ be a lcsc, and suppose $X$ is a compact $G$-space that admits a fully supported $G$-tight probability measure $\nu$. If $G\act (X, \nu)$ is Zimmer-amenable then $G\act X$ is a topological boundary action.
\end{prop}
\begin{proof}
For any $\eta\in\pr(X)$, the Poisson transform $\cP_\eta$ is a $G$-equivariant unital positive map from $C(X)$ into $C_b(G)$. Since $G\act (X, \nu)$ is Zimmer-amenable, there is a $G$-equivariant unital positive map $\psi_0$ from $C_b(G)$ to $L^\infty(X, \nu)$. By tightness, $\psi_0\circ\cP_\eta = \id$, which implies $\cP_\eta$ is isometric for every $\eta\in\pr(X)$. By~\cite[Proposition 1.1]{Azen-book}, it follows that the weak* closure of the $G$-orbit of $\eta$ in $\pr(X)$ contains $X$, hence the action $G\act X$ is minimal and strongly proximal.
\end{proof}

\subsection*{Other examples}
All examples of tight measure-classes presented so far occur in the setup of boundary actions. However, examples of tight measure-classes appear in more general settings. 
Below, we give examples of actions supporting tight measure-classes that are not boundary actions in a topological or measurable sense.

\begin{example}\label{ex:non-ergodic}
Let $G_1$ and $G_2$ be two lcsc groups, and for $i=1, 2$ let $X_i$ be a compact $G_i$-space and $\nu_i\in \pr(X_i)$ an ergodic non-trivial tight measure-class. Consider the action of the product group $G=G_1 \times G_2$ on the disjoint union $X=X_1\dist X_2$ where $G_1$ and $G_2$ act trivially on $X_2$ and $X_1$ respectively. Then the measure $\nu=\frac{1}{2}\nu_1 +\frac{1}{2}\nu_2$ is a $G$-tight measure-class on $X=X_1 \dist X_2$. 
\end{example}

In particular, the above (somewhat superficial) example shows that unlike the case of boundary actions, neither ergodicity of the measure nor the minimality of the topological action follow from the tightness of the measure-class. Furthermore, we conclude another difference of tight measure-classes to the case of boundaries.

\begin{cor}\label{cor:rig-not->factors}
Tightness of measures does not, in general, pass to (measurable or continuous) factors.
\end{cor}
\begin{proof}
The tight action from Example~\ref{ex:non-ergodic} clearly has a continuous 
factor, namely, the space with 2 points, equipped with the 
uniform measure, which is not tight by Lemma~\ref{lem:rigincl+inv->trv}.
\end{proof}

We continue with more interesting examples of rigid actions which are not boundaries. In particular, we give examples of purely atomic tight measure-classes on  non-minimal, non-strongly proximal spaces.

Following \cite{BekKal20} we say an open subgroup $H$ of $G$ has \emph{the spectral gap property} if $\delta_H$ is the unique $H$-invariant mean on $\ell^\infty(G/H)$.

\begin{thm}\label{thm:specgapprop-equiv}
Let $H$ be an open subgroup of $G$. Then TFAE
\begin{enumerate}
\item 
$H$ has the spectral gap property.
\item
Any $\nu\in \pr(G/H)$ with full support considered as a probability on the Stone-\v{C}ech compactification $X=\beta(G/H)$ is $G$-tight.
\item
The inclusion $\inc{\ell^\infty(G/H)}{\ell^\infty(G/H)}$ is $G$-tight.
\item
The inclusion $\inc{\ell^\infty(G/H)}{\cB(\ell^2(G/H))}$ is $G$-tight.
\end{enumerate}
\begin{proof}
$(1)\Rightarrow(2): $ Note that $C(X) = \ell^\infty(G/H) = L^\infty(X, \nu)$. Thus we need to show that the only $G$-equivariant unital positive map on $\ell^\infty(G/H)$ is the identity map. Let $\psi\colon  \ell^\infty(G/H)\to \ell^\infty(G/H)$ be such a map. Then $\psi^*(\delta_{H})$ is an $H$-invariant mean on $\ell^\infty(G/H)$, hence $\psi^*(\delta_{H}) = \delta_{H}$ by uniqueness. Since $\psi^*$ is $G$-equivariant it follows $\psi^*(\delta_{gH})=\delta_{gH}$ for all $g\in G$. Since $G/H$ is dense in $\beta(G/H)$ and $\psi^*$ is weak* continuous it follows $\psi^* = \id$, and consequently $\psi=\id$.
\\[1ex]
$(2) \Longleftrightarrow (3):$ Note that $(2)$ and $(3)$ are just two formulations of the same thing, by the definition of the Stone-\v{C}ech compactification.
\\[1ex]
$(3)\Rightarrow(1): $ Assume $H$ does not have the spectral gap property, and let $\phi$ be an $H$-invariant mean on $\ell^\infty(G/H)$ different from $\delta_H$. Then the Poisson transform $\cP_\phi\colon \ell^\infty(G/H)\to \ell^\infty(G)$ is indeed mapped into $\ell^\infty(G/H)$, and we observe that $\cP_\phi^*(\delta_H) = \phi$. In particular, $\phi\neq\id$, which shows $H$ does not have the spectral gap property.
\\[1ex]
$(3)\Rightarrow(4): $ The canonical conditional expectation $\cB(\ell^2(G/H)) \to \ell^\infty(G/H)$ is $G$-equivariant and faithful. Hence, it follows from Lemma~\ref{lem:faith-cond-exp-->G-rigid} that the inclusion $\inc{\ell^\infty(G/H)}{\cB(\ell^2(G/H))}$ is $G$-tight.
\\[1ex]
$(4)\Rightarrow(3): $ This follows from Proposition~\ref{prop:tight->suablg->tight}.
\end{proof}
\end{thm}
%
%
Examples of subgroups with the spectral gap property include the following (see \cite{BekKal20} for more on these examples):
\begin{itemize}
\item 
$SL_n(\bZ)\leq SL_{n+1}(\bZ)$ for $n\geq 2$;
\item
$H\leq H*K$ for any non-amenable $H$ and any $K$ with $|K|\geq 3$.
\end{itemize}

Denote by $\Sub^o_{\rm sg}(G)$ the set of open subgroups of $G$ with the spectral gap property.

In \cite[Theorem A]{BekKal20}, a representation rigidity result was proved for subgroups $\Lambda\in \Sub^o_{\rm sg}(\G)$
of a discrete group $\G$; namely, it was shown that if $\Upsilon$ is a self-commensurated subgroup of $\G$ such that the quasi-regular representation $\lambda_{\G/\Upsilon}$ is weakly equivalent to $\lambda_{\G/\Lambda}$, then $\Upsilon$ is conjugate to $\Lambda$.

Using the self-tightness property of subgroup $H\in \Sub^o_{\rm sg}(G)$, we see below that they satisfy a much stronger representation rigidity with themselves.

\begin{defn}
Let $\pi$ and $\sigma$ be continuous unitary representations of a lcsc group $G$. We say $\pi$ is \emph{barely contained} in $\sigma$, denoted $\pi\llcurly\sigma$, if there is a $G$-equivariant ucp map from $\cB(H_\sigma)$ to $\cB(H_\pi)$.
We say $\pi$ and $\sigma$ are barely equivalent, denoted $\pi\overset{\scriptsize\text{b}}{\sim}\sigma$, if $\pi\llcurly\sigma$ and $\sigma\llcurly\pi$.
\end{defn}
We observe that if $\pi$ is weakly contained in $\sigma$, then $\pi$ is also barely contained in $\sigma$, but the converse is very far from being the case. For instance, if $G$ is amenable, then $\pi\overset{\scriptsize\text{b}}{\sim}\sigma$ for any $\pi$ and $\sigma$.

\begin{thm}
Let $H, L \in \Sub^o_{\rm sg}(G)$. Then $\lambda_{G/H}\overset{\scriptsize{\rm b}}{\sim}\lambda_{G/H}$ iff $H$ and $L$ are conjugate in $G$.
\begin{proof}
Let $\tilde\varphi\colon \cB(\ell^2(G/H))\to \cB(\ell^2(G/L))$ be a $G$-equivariant ucp map. Restricting $\tilde\varphi$ to $\ell^\infty(G/H)$ and then compose it with the canonical conditional expectation $\cB(\ell^2(G/L))\to \ell^\infty(G/L)$, we get a $G$-equivariant unital positive map $\varphi\colon \ell^\infty(G/H)\to \ell^\infty(G/L)$. Similarly, we get a $G$-equivariant unital positive map $\psi\colon \ell^\infty(G/L)\to \ell^\infty(G/H)$. 

By Theorem~\ref{thm:specgapprop-equiv}, both $\ell^\infty(G/H)$ and $\ell^\infty(G/L)$ are $G$-self-tight, hence $\psi\circ\varphi=\id_{\ell^\infty(G/H)}$ and $\varphi\circ\psi=\id_{\ell^\infty(G/L)}$. It follows $\varphi$ and $\psi$ are isometric linear isomorphisms, hence von Neumann algebra isomorphisms. Thus, there is a $G$-equivariant bijection $G/H\to G/L$. This implies there exists $g\in G$ such that $L=gHg^{-1}$.
\end{proof}
\end{thm}

\medskip

Another interesting class of examples of tight measure-classes appear as atomic measures on the orbit of ``parabolic-type points'' as follows. 

\begin{thm}\label{thm:parsubgrp->rigmsre}
Let $\G$ be a countable discrete group and $\Lambda$ a subgroup of $\G$. Assume there exists a minimal compact $\G$-space $X$ containing a $\Lambda$-fixed point $x_0$ such that $\delta_{x_0}$ is the unique $\Lambda$-invariant probability measure on $X$. Then 
\begin{itemize}
\item[(i)]
The measure $\nu := \sum_{n=1}^\infty\frac{1}{2^n}\delta_{g_nx_0}$ is a tight measure on $X$, where $\{g_n\}_{n\in\bN}$ is a complete set of representatives of $\G/\Lambda$.
\item[(ii)]
The inclusion given by the Poisson transform $\cP_{\delta_{x_0}}(C(X)) \subset \ell^\infty(\G/\Lambda)$ is a $\G$-tight inclusion.
\end{itemize}
\end{thm}

\begin{proof}
Note that $L^\infty(X, \nu)$ is $\G$-equivariantly isomorphic to $\ell^\infty(\G/\Lambda)$, and this isomorphism on $C(X)$ is the Poisson transform $\cP_{\delta_{x_0}}$. So, we only need to prove (ii). For this, we argue similarly as in the proof of Theorem~\ref{thm:specgapprop-equiv}.
Let $\psi\colon  C(X)\to \ell^\infty(\G/\Lambda)$ be a $\G$-equivariant unital positive map. Then $\psi^*(\delta_{\Lambda})$ is an $\Lambda$-invariant state on $C(X)$, hence the point evaluation at $x_0$ by the uniqueness assumption. Since $\psi^*$ is $\G$-equivariant it follows $\psi^*(\delta_{g\Lambda})=\delta_{gx_0}$ for all $g\in \G$. This shows $\psi = \cP_{x_0}$, hence (ii) follows.
\end{proof}

\begin{cor}\label{cor:parsubgrp->tight}
Let $\G\act X$, $\Lambda\leq \G$ and $x_0\in X$ be as in the statement of Theorem~\ref{thm:parsubgrp->rigmsre}. 
Then the inclusion $\inc{\cP_{\delta_{x_0}}(C(X))}{\cB(\ell^2(\G/\Lambda))}$ is $\G$-tight.
\end{cor}
\begin{proof}
The canonical conditional expectation $\cB(\ell^2(\G/\Lambda)) \to \ell^\infty(\G/\Lambda)$ is $\G$-equivariant and faithful. Hence, it follows from Theorem~\ref{thm:parsubgrp->rigmsre} and Lemma~\ref{lem:faith-cond-exp-->G-rigid} that the inclusion $\inc{\cP_{\delta_{x_0}}(C(X))}{\cB(\ell^2(\G/\Lambda))}$ is $\G$-tight.
\end{proof}

\section{Applications: intermediate objects}\label{sec:appl-int-obj}

In this section, we use properties of tight inclusions and tight measure-classes to prove certain rigidity results concerning intermediate operator algebras associated with tight inclusions. 

Let $\inc{\sA}{\sB}$ be an inclusion of objects in $\Op_G$. By an intermediate object (for the inclusion) we mean a $\sD\in \Op_G$ such that $\incInt{\sA}{\sD}{\sB}$  
which is also assumed to be a $G$-von Neumann algebra in case $\sB$ is.

Let $\sB\in\Op_G$, and let $\sA, \sC\subset \Op_G$ be $G$-invariant $C^*$-subalgebras of $\sB$.  
We write $\sB=\sA\vee\sC$ if $\sB$ is the object generated by $\sA$ and $\sC$; this means if $\sB$ is a $G$-$C^*$-algebra then it is the $C^*$-algebra generated by $\sA$ and $\sC$, and if $\sB$ is a $G$-von Neumann algebra then it is the von Neumann algebra generated by $\sA$ and $\sC$.

\begin{defn}
We say the inclusion $\inc{\sC}{\sB}$ is \emph{co-tight} if there is a $G$-tight inclusion $\inc{\sA}{\sB}$ such that $\sB = \sA \vee \sC$.
\end{defn}

\begin{example}
Let $\bF_2\act (Z, m)$ be an ergodic probability measure preserving (pmp) action. Then for any generating $\mu\in \pr(\bF_2)$ and any $\mu$-stationary $\nu\in\pr(\partial\bF_2)$, then using Theorem~\ref{thm:USB->rigid} and Lemma~\ref{lem:faith-cond-exp-->G-rigid}, one can show that the inclusion $L^\infty(Z, m) \subset L^\infty(Z\times \partial\bF_2, m\times \nu)$ is co-tight. 
\end{example}

\begin{example}\label{ex:parsubgrp->co-tight}
If $\Lambda$ is a subgroup of $\G$ with spectral gap property, or if it is a subgroup as in the statement of Theorem~\ref{thm:parsubgrp->rigmsre},
then it follows from Corollary~\ref{cor:parsubgrp->tight} that the inclusion $\inc{L\G}{\G\ltimes \ell^\infty(\G/\Lambda)}$ is co-tight, where $L\G$ denotes the (left) group von Neumann algebra of $\G$, that is, the von Neumann algebra generated by the left regular representation of $\G$.
\end{example}

Recall the notion of \emph{$G$-rigid extension} in the sense of Hamana \cite{Ham85}: if ${\sA}\subset{\sB}$ is an inclusion of $G$-$C^*$-algebras, then $\sB$ is said to be a $G$-rigid extension of $\sA$ if the identity map on $\sB$ is the unique $G$-equivariant ucp extension of the identity map on $\sA$.

\begin{prop}
Let $G$ be a lcsc group, and $\inc{\sC}{\sB}$ a co-tight inclusion of $G$-$C^*$-algebras. Then $\sB$ is a $G$-rigid extension of $\sC$.
\begin{proof}
By assumptions, there exists a $G$-tight inclusion $\inc{\sA}{\sB}$ such that $\sB=\sA\vee\sC$. Assume $\psi\colon \sB\to\sB$ is a $G$-equivariant ucp map such that $\psi|_\sC=\id_\sC$. By tightness, we also have $\psi|_\sA=\id_\sA$. Since $\psi$ is ucp, it follows $\psi|_{\sA\vee\sC}=\id|_{\sA\vee\sC}$, and this implies the claim.
\end{proof}
\end{prop}

The following two simple observations show how the notion of (co-)tightness will be used in this section to prove certain applications, namely in proving the lack of (weakly) Zimmer amenable intermediate objects in several setups.

\begin{lem}\label{lem:cotight->noRelZimInt}
Let $G$ be a lcsc group, and let $\inc{\sC}{\sB}$ be a co-tight inclusion of objects in $\Op_G$. If there is a $G$-equivariant ucp map from $\sB$ to $\sC$, then $\sC=\sB$.
\begin{proof}
Assume $\psi\colon \sB\to\sC$ is a $G$-equivariant ucp map. Since $\inc{\sC}{\sB}$ is co-tight, there is a $G$-tight inclusion $\inc{\sA}{\sB}$ such that $\sB$ is generated by $\sA$ and $\sC$. By tightness of $\inc{\sA}{\sB}$, the map $\psi$ restricts to identity on $\sA$, thus $\sA\subset \sC$. Again, using co-tightness, $\sB=\sA\vee\sC\subset\sC$.
\end{proof}
\end{lem}

\begin{lem}\label{lem:cotight->noWZim}
Let $G$ be a lcsc group. A co-tight inclusion of objects in $\Op_G$ admits no weakly Zimmer-amenable proper intermediate objects.
\begin{proof}
This follows from Lemma~\ref{lem:cotight->noRelZimInt} and the obvious fact that any intermediate object for a co-tight inclusion is also co-tight.
\end{proof}
\end{lem}

Below we use standard terminology from ergodic theory of joinings and relatively measure-preserving extensions. In particular, we use the terminology introduced and used in \cites{FursGlas10}.
We recall that a $G$-space $(\tilde{Y},\tilde \eta)$ is a \emph{joining} of the $G$-spaces $(Y,\eta)$ and $(X,\nu)$ if both $L^\infty(Y,\eta)$ and $L^\infty(X,\nu)$ are embedded $G$-equivarently in $L^\infty(\tilde Y,\tilde \eta)$ and the latter is the von Neumann algebra generated by these two sub-algebras.

A factor map $\phi\colon (\tilde{Y},\tilde{\eta})\to (X,\nu)$ is said to be \emph{relatively measure-preserving} if the canonical conditional expectation is $G$-equivariant. The more common definition in ergodic theory usually uses the adjoint setup and requires equivariance of the disintegration map.

\begin{prop}\label{prop:Zamen-int-fact}
Let $G$ be a lcsc group, $X$ a compact $G$-space, and let $\nu$ be a $G$-tight measure on $X$. Assume that $(\tilde Y, \tilde \eta)$ is a relatively measure-preserving extension of $(X, \nu)$, and $\phi\colon(\tilde Y, \tilde \eta)\to(Y, \eta)$ is a $G$-factor map such that $(\tilde{Y},\tilde \eta)$ is a joining of $(Y,\eta)$ and $(X,\nu)$. If $(\tilde Y, \tilde \eta) \xrightarrow{\varphi} (Z,\om) \xrightarrow{\psi} (Y, \eta)$ are $G$-factor maps such that $\psi\circ\varphi=\phi$ and $(Z,\om)$ is weakly Zimmer amenable then $(\tilde Y, \tilde \eta) \overset{\varphi}{\cong} (Z,\om)$.
\begin{proof}
First, note that the maps $\phi$, $\varphi$, and $\psi$ yield canonical embeddings of corresponding $L^\infty$-algebras, which we identify with their images under these embeddings.

Since $(\tilde Y, \tilde \eta)$ is a relatively measure-preserving extension of $(X, \nu)$, there is a $G$-equivariant faithful conditional expectation $L^\infty(\tilde Y, \tilde \eta) \to L^\infty(X, \nu)$. Hence, by Lemma~\ref{lem:faith-cond-exp-->G-rigid}, the inclusion $\inc{C(X)}{L^\infty(\tilde Y, \tilde \eta)}$ is $G$-tight. 

Since $\tilde \eta$ is a joining of $\nu$ and $\eta$, we have $L^\infty(\tilde Y, \tilde \eta) = L^\infty(X, \nu)\vee L^\infty(Y, \eta)= C(X)\vee L^\infty(Y, \eta)$, therefore the inclusion $\inc{L^\infty(Y, \eta)}{L^\infty(\tilde Y, \tilde \eta)}$ is co-tight. Thus, if $L^\infty(Z,\om)$ is weakly Zimmer amenable, it follows by Lemma~\ref{lem:cotight->noWZim} that $L^\infty(\tilde Y, \tilde \eta) = L^\infty(Z,\om)$, and this completes the proof.
\end{proof}
\end{prop}

In \cite{FursGlas10}, Furstenberg and Glasner, in their attempt to formulate an abstract version of Nevo--Zimmer's structure theorem \cite{nevo2020structure} for stationary actions of higher rank semisimple Lie groups, they introduced and studied the notion of \emph{standard covers}. 
Let us recall that a $(G,\mu)$-space $(\tilde{Y},\tilde{\eta})$ is called \emph{standard} if it is a relatively measure-preserving extension of a $\mu$-boundary. 
The Structure Theorem of Furstenberg--Glasner~\cite[Theorem 4.3]{FursGlas10} states that for each  $(G,\mu)$-space $(Y,\eta)$, there exists a unique standard space $(\tilde{Y},\tilde{\eta})$, called the {standard cover of $(Y,\eta)$}, with the property that there exists a $\mu$-boundary $(X,\theta)$ which is a relatively measure-preserving factor of $(\tilde{Y},\tilde{\eta})$, and such that $(\tilde{Y},\tilde \eta)$ is a joining of $(Y,\eta)$ and $(X,\theta)$. Using this terminology, Nevo--Zimmer's theorem~\cite{nevo2020structure} states that for higher-rank semisimple Lie groups many stationary actions are standard. Moreover, in this setup, not only that the spaces are standard, but the $\mu$-boundary $(X,\theta)$ admits a USB model.

\begin{thm}\label{thm:Zamen-int-fact}
Let $G$ be a lcsc group and $\mu\in\pr(G)$ an admissible  measure. Let $(Y, \eta)$ be a $(G, \mu)$-space with the standard cover $(\tilde Y, \tilde \eta)$ realized by the factor maps $\phi\colon (\tilde Y, \tilde \eta)\to (Y, \eta)$ and $\phi'\colon (\tilde Y, \tilde \eta)\to (X, \theta)$, where $(X, \theta)$ is a $\mu$-boundary and $\phi'$ is relatively measure-preserving. Let $(Z,\om)$ be a weakly Zimmer amenable $(G, \mu)$-space, and $(\tilde Y, \tilde \eta) \xrightarrow{\varphi} (Z,\om) \xrightarrow{\psi} (Y, \eta)$ are $G$-maps such that $\psi\circ\varphi=\phi$. If, either
\begin{itemize}
\item [(i)\,]
$(X, \theta)$ has a USB model, or
\item [(ii)]
the Furstenberg-Poisson Boundary of $(G, \mu)$ has a USB model,
\end{itemize}
then $(\tilde Y, \tilde \eta) \overset{\varphi}{\cong} (Z,\om)$.
\end{thm}
\begin{proof}
(i)\ Without loss of generality, we may assume $(X, \theta)$ is $\mu$-USB. Then $\theta$ is a tight measure by Theorem~\ref{thm:USB->rigid}, and therefore the assertion follows from Proposition~\ref{prop:Zamen-int-fact}.\\
(ii)\ Let $(B, \nu)$ denote a USB model of the Furstenberg-Poisson Boundary of $(G, \mu)$. The space $(\tilde Y, \tilde \eta)$ is weakly Zimmer amenable as it is a $G$-extension of a weakly Zimmer amenable space $(Z,\omega)$, so there is a $G$-equivariant ucp map from $C(B)$ to $L^\infty(\tilde Y, \tilde \eta)$. By Lemma~\ref{lem:quasi-factor}, there exists a measurable $G$-equivariant map $\rho\colon\tilde Y \to \pr(B)$. Since $(B, \nu)$ is USB, it follows from \cite[Corollary 2.10(a)]{Margulis-book-91} that $\rho$ is mapped into delta measures, hence a $G$-equivariant factor map $\rho\colon(\tilde Y, \tilde \eta) \to (B, \nu)$ (cf. \cite[Theorem 9.2(1)]{NevSag13}). 
We claim that $\phi'={\rm\bf bnd}\circ\rho$, where ${\rm\bf bnd}\colon (B, \nu)\to (X, \theta)$ is the canonical $G$-factor map. The claim then implies $\rho$ is relatively measure-preserving, hence put us in the setup of Proposition~\ref{prop:Zamen-int-fact}, and completes the proof. The claim, indeed, holds in more generality where USB assumption is not needed. We state this in the following lemma.
\end{proof}
\begin{lemma}\label{lem:structure-proximal}
Let $G$ be a lcsc group and let $\mu\in\pr(G)$ be an admissible measure. Let $(X, \theta)$ be a $(G, \mu)$-boundary, let $(Y, \eta)$ be a $(G, \mu)$-space, and $\rho\colon (Y,\eta)\to (X,\theta)$ be a relatively measure-preserving factor map. Then $\rho^*$ is the unique normal ucp map from $L^\infty(X,\theta)$ to $L^\infty(Y, \eta)$.
\begin{proof}
Let $\Phi:L^\infty(X,\theta)\to L^\infty(Y, \eta)$ be a normal $G$-equivariant ucp map.  
Denote by $\bE\colon L^\infty(Y, \eta) \to \rho^*(L^\infty(X,\theta))$ the canonical conditional expectation, which is $G$-equivariant by the assumption. Then $\bE\circ\Phi$ is a normal unital positive $G$-equivariant map from  $L^\infty(X,\theta)$ to $\rho^*(L^\infty(X,\theta))$, hence coincides with $\rho^*$ by Lemma~\ref{lem:normal-self-tight}. Since $\bE$ is faithful, it follows $\rho^*=\Phi$.
\end{proof}
\end{lemma}

\begin{remark}
Note that Lemma~\ref{lem:structure-proximal} also implies in the proof of part (ii) above, that the map ${\rm\bf bnd}$ is relatively measure-preserving. It is not hard to see that any relatively measure-preserving factor map between measurable boundaries is an isomorphism. Thus, the assumption (ii) in fact implies $(X, \theta)= (B, \nu)$.
\end{remark}

In the noncommutative setting, co-tight inclusions appear naturally in the setting of covariant representations of actions of tight inclusions. 

\begin{prop}\label{prop:no-amen-inter-obj}
Let $\G$ be a discrete group and $\sA\in\Op_{\G}$. Suppose that $(\pi,\rho)$ is a covariant representation of $(\G, \sA)$ such that the inclusion $\inc{\rho(\sA)}{ \G\ltimes_{\scriptscriptstyle\tiny\pi}^{\scriptscriptstyle\tiny\rho}\sA}$ is $\G$-tight. Let $\sD\in \Op_{\G}$ be an intermediate object for the inclusion $\inc{C^*_\pi(\G)}{\G\ltimes_{\scriptscriptstyle\tiny\pi}^{\scriptscriptstyle\tiny\rho}\sA}$. If there is a $\G$-equivariant ucp map from $\G\ltimes_{\scriptscriptstyle\tiny\pi}^{\scriptscriptstyle\tiny\rho}\sA$ to $\sD$, then $\sD=\G\ltimes_{\scriptscriptstyle\tiny\pi}^{\scriptscriptstyle\tiny\rho}\sA$. 

In particular, if $\sD$ is either weakly Zimmer-amenable or injective, then $\sD=\G\ltimes_{\scriptscriptstyle\tiny\pi}^{\scriptscriptstyle\tiny\rho}\sA$.
\begin{proof}
Note that since the inclusion $\inc{\rho(\sA)}{ \G\ltimes_{\scriptscriptstyle\tiny\pi}^{\scriptscriptstyle\tiny\rho}\sA}$ is $\G$-tight and $\G\ltimes_{\scriptscriptstyle\tiny\pi}^{\scriptscriptstyle\tiny\rho}\sA$ is generated by $\rho(\sA)$ and $C^*_\pi(\G)$, the inclusion $\inc{C^*_\pi(\G)}{ \G\ltimes_{\scriptscriptstyle\tiny\pi}^{\scriptscriptstyle\tiny\rho}\sA}$ is co-tight. Thus, the claim follows from Lemma~\ref{lem:cotight->noRelZimInt} and the fact that any intermediate object for a co-tight inclusion is also co-tight.

The case of weakly Zimmer-amenable $\sD$ follows from Lemma~\ref{lem:cotight->noWZim}.
If $\sD$ is injective, then there is a conditional expectation $\bE\colon \G\ltimes_{\scriptscriptstyle\tiny\pi}^{\scriptscriptstyle\tiny\rho}\sA\to\sD$, which is automatically $\G$-equivariant since $\G$ is in the multiplicative domain of $\bE$. Thus, the claim follows from the above.
\end{proof}
\end{prop}

We denote by $C^*_{\rm red}(\G)$ the reduced $C^*$-algebra of $\G$, that is, the $C^*$-algebra generated by the left regular representation of $\G$ on $\ell^2(\G)$.
\begin{cor}\label{cor:no-amen-inter-obj-crssed}
Let $\G$ be a discrete group, let $\sA\in\Op_{\G}$ be $\G$-self-tight, and let $\sD$ be an intermediate object for the inclusion $\inc{C^*_{\rm red}(\G)}{\G\ltimes \sA}$. If $\sD$ is either weakly Zimmer-amenable or injective, then $\sD=\G\ltimes \sA$.
\begin{proof}
By Corollary~\ref{cor:tight->crsprdt-tight}, the inclusion $\inc{\sA}{\G\ltimes\sA}$ is $\G$-tight.  Hence, the result follows from Proposition~\ref{prop:no-amen-inter-obj}.
\end{proof}
\end{cor}

The above corollary applies, for example, in the case $\sA=C(X)$ where $X$ is a topological boundary (see Example~\ref{top_bondaries}).

\begin{cor}\label{cor:no-amen-inter-obj-vn}
Let $\G$ be a discrete group, $\sB$ a $\G$-von Neumann algebra, and $\sA$ a weak*-dense $\G$-$C^*$-subalgebra of $\sB$ such that $\inc{\sA}{\sB}$ is $\G$-tight. Suppose that $\sD$ is an intermediate $\G$-von Neumann algebra for the crossed product, that is, $L\G\subset \sD\subset  \G\ltimes \sB$. If $\sD$ is either weakly Zimmer-amenable or injective, then $\sD=\G\ltimes \sB$.
\end{cor}

\begin{proof}
By Corollary~\ref{cor:tight->crsprdt-tight}, the inclusion $\sA\subset \G\ltimes\sB$ is also $\G$-tight. Thus, the result follows from Proposition~\ref{prop:no-amen-inter-obj}.
\end{proof}

Next, we prove noncommutative generalizations of earlier results of this section on intermediate objects in the discrete group case, and apply them to prove a maximal injectivity result (Corollary~\ref{cor:max-inj-vn}).

We begin with an extension of Lemma~\ref{lem:cotight->noWZim}.
\begin{lem}\label{lem:nc-co-tight->noZamen-int}
Let $\G$ be a countable discrete group, and $\sC\subset\sB$ a co-tight inclusion of objects in $\Op_{\G}$. Then the inclusion $\G\ltimes\sC\subset\G\ltimes\sB$ is co-tight, hence, in particular, has no injective or weakly Zimmer amenable intermediate proper objects.
\begin{proof}
By co-tightness, there is a $\G$-tight inclusion $\sA\subset\sB$ such that $\sB=\sA\vee\sC$. Then the inclusion $\sA\subset\G\ltimes\sB$ is tight by Corollary~\ref{cor:tight->crsprdt-tight}, which then implies the inclusion $\G\ltimes\sC\subset\G\ltimes\sB$ is co-tight. Hence, arguing similarly as in the proof of Proposition~\ref{prop:no-amen-inter-obj}, we conclude that the inclusion has no injective or weakly Zimmer amenable intermediate proper objects.
\end{proof}
\end{lem}

The following is the noncommutative extension of Proposition~\ref{prop:Zamen-int-fact}.

\begin{prop}\label{prop:nc-no-Z-int-gen}
Let $\G$ be a discrete group, $X$ a compact $\G$-space, and let $\nu$ be a $\G$-tight measure on $X$. Assume that $(\tilde Y, \tilde \eta)$ is a relatively measure-preserving extension of $(X, \nu)$, and $(Y,\eta)$ a $\G$-factor of $(\tilde Y, \tilde \eta)$ such that $(\tilde{Y},\tilde \eta)$ is a joining of $(Y,\eta)$ and $(X,\nu)$. Suppose $\G\ltimes L^\infty(Y, \eta)\subseteq M \subseteq \G\ltimes L^\infty(\tilde Y, \tilde \eta)$ is an inclusion of von Neumann algebras. If $M$ is injective, then $M = \G\ltimes L^\infty(\tilde Y, \tilde \eta)$.
\begin{proof}
As seen in the proof of Proposition~\ref{prop:Zamen-int-fact}, the assumptions imply that the inclusion $L^\infty(Y, \eta)\subseteq L^\infty(\tilde Y, \tilde \eta)$ is co-tight. Thus, the result follows from Lemma~\ref{lem:nc-co-tight->noZamen-int}.
\end{proof}
\end{prop}

The setup of the Proposition~\ref{prop:nc-no-Z-int-gen} is quite general, provides large classes of examples. We single out two interesting cases in the following.

\begin{cor}\label{cor:no-inj-int-two-case}
Let $\G$ be a discrete group and $\mu\in\pr(\G)$ a generating measure such that the Furstenberg-Poisson Boundary $(B, \nu)$ of $(\G, \mu)$ has a uniquely stationary compact model.
Then the conclusion of Proposition~\ref{prop:nc-no-Z-int-gen} holds in the following two cases, for any $\G$-factor map $(\tilde Y, \tilde \eta)\to (Y, \eta)$:
\begin{itemize}
\item [(i)\,]
$(\tilde Y, \tilde \eta)$ is Zimmer amenable and is the standard cover of $(Y, \eta)$; or
\item [(ii)]
$(\tilde Y, \tilde \eta)= (B\times Z, \nu\times m)$ and $(Y, \eta)= (B, \nu)$, with the canonical factor map, for some ergodic pmp action $\G\act (Z, m)$.
\end{itemize}
\begin{proof}
Under the assumptions of part (i), it was shown in the proof of Theorem~\ref{thm:Zamen-int-fact}(ii) that $(\tilde Y, \tilde \eta)$ is a relatively measure-preserving extension of $(B, \nu)$, and also is a joining of $(Y, \eta)$ and $(B, \nu)$. Since $\nu$ is $\G$-tight by Theorem~\ref{thm:USB->rigid}, the assertion (i) follows from Proposition~\ref{prop:nc-no-Z-int-gen}.\\
Now, assume (ii). Then $(\tilde Y, \tilde \eta)$ is a relatively measure-preserving extension of $(B, \nu)$, indeed the integration of the $Z$-component with respect to $m$ is the canonical conditional expectation $L^\infty(B \times Z, \nu \times m)\to L^\infty(B, \nu)$, which is obviously $\G$-equivariant. Also, $(\tilde Y, \tilde \eta)$ is clearly a joining of $(Y, \eta)$ and $(B, \nu)$. Thus, assertion (ii) also follows from Proposition~\ref{prop:nc-no-Z-int-gen}. 
\end{proof}
\end{cor}

The above statements, can be considered as `minimal ambient injectivity' results, and therefore, since the commutant of an injective von Neumann algebra is also injective (regardless of the representation), by taking commutants in the above inclusions, we obtain (new) examples of maximal injective von Neumann subalgebras. 

In particular, in the product space case as in part (ii) of Corollary~\ref{cor:no-inj-int-two-case}, we get the following.

Given a non-singular action $\G\curvearrowright (X, \nu)$, we have an action $\G\curvearrowright \cB(L^2(X, \nu))$ by inner automorphisms associated with the corresponding Koopman representation.

\begin{cor}\label{cor:max-inj-vn}
Let $\G$ be a discrete group and $\mu\in\pr(G)$ a generating measure such that the Furstenberg-Poisson Boundary $(B, \nu)$ of $(\G, \mu)$ has a uniquely stationary compact model. Let $\G\curvearrowright (Z, m)$ be a measure-preserving action. 
Then the von Neumann algebra $\G\ltimes L^\infty(B \times Z , \nu \times m)$ is maximal injective in $\G\ltimes \big(\cB(L^2(B, \nu))\overline{\otimes}L^\infty(Z, m)\big)$.
\begin{proof}
Assume that $N$ is an injective von Neumann algebra, and that we have $\G\ltimes L^\infty(B \times Z , \nu \times m) \subseteq N \subseteq \G\ltimes \left(\cB(L^2(B, \nu))\overline{\otimes}L^\infty(Z, m)\right)$. Taking commutants in $B\big(\ell^2(\G)\otimes L^2(B \times Z , \nu \times m)\big)$, we get (see e.g. \cite[Proposition V.7.14]{Tak1}) the inclusion of von Neumann algebras 
\[
\G\ltimes L^\infty(Z , m)  \subseteq N' \subseteq \G\ltimes L^\infty(B \times Z , \nu \times m) ,
\]
and $N'$ is injective. Thus, $N' = \G\ltimes L^\infty(B \times Z , \nu\times m)$ by Proposition~\ref{prop:nc-no-Z-int-gen}, and hence $N = \G\ltimes L^\infty(B \times Z , \nu\times m)$.
\end{proof}
\end{cor}
The above corollary can be of course stated in the more general setup of Proposition~\ref{prop:nc-no-Z-int-gen}, but even a special case of the corollary, where $Z$ is the trivial space, yields a stronger and more general version of Suzuki's maximal injectivity result \cite[Corollary 3.8]{Suz18}. Our proof strengthens his result as we do not need splitting property for the crossed products (\cite[Theorem 3.6]{Suz18}), which only holds for extensions of free actions. Furthermore, and more importantly, our proof also avoids the use of Margulis' Factor Theorem, which is a deep result concerning higher rank lattices. In particular, it shows that the source of such rigidity of intermediate von Neumann algebras, is not higher rank, but the much wider phenomenon of tightness. In fact, we obtain a large class of new examples of maximal injective von Neumann algebras. As remarked before, examples of Furstenberg-Poisson Boundary actions with a uniquely stationary compact model, include: the actions of hyperbolic groups on the Gromov boundary, linear groups on flag varieties, mapping class groups on the Thurston boundary, and $\rm{Out}(\bF_n)$ on the boundary of the outer space, all for suitable $\mu$'s.
\begin{remark}
Similar methods as in the above, but applied in the case of tight inclusions as in Theorem~\ref{thm:parsubgrp->rigmsre} for a subgroup $\Lambda\leq \Gamma$ that is, in addition, a maximal amenable subgroup, proves maximal injectivity of $L\Lambda$ in $L\Gamma$.
Such examples include for instance, maximal abelian subgroups of free groups. The result of Boutonnet--Carderi \cite[Theorem~A]{BoutCar15} covers these examples, but our methods offer a new approach which may result in new examples in this setup.
\end{remark}

\end{document}